\documentclass[12pt]{amsart}
\usepackage{amssymb,latexsym}
\usepackage{amsfonts}
\usepackage{amsmath}
\usepackage[colorlinks,linkcolor=blue,anchorcolor=blue,citecolor=blue]{hyperref}
\usepackage{algorithm}
\usepackage{enumerate}
\usepackage{algpseudocode}
\usepackage{verbatim}
\usepackage{graphicx}

\newcommand\F{{\mathbb F}}
\newcommand\Z{{\mathbb Z}}
\newcommand\A{{\mathbb A}}
\newcommand\N{{\mathbb N}}
\newcommand\cA{{\mathcal A}}

\newtheorem{theorem}{Theorem}[section]
\newtheorem{lemma}[theorem]{Lemma}

\newtheorem{corollary}[theorem]{Corollary}

\theoremstyle{definition}
\newtheorem{definition}[theorem]{Definition}
\newtheorem{remark}[theorem]{Remark}

\numberwithin{equation}{section}

\begin{document}

\title[Congruence preserving functions]{Congruence preserving functions in the residue class rings of polynomials over finite fields}

\author{Xiumei Li}
\address{School of Mathematical Sciences, Qufu Normal University, Qufu Shandong, 273165, China}
\email{lxiumei2013@mail.qfnu.edu.cn}

\author{Min Sha}
\address{Department of Computing, Macquarie University, Sydney, NSW 2109, Australia}
\email{shamin2010@gmail.com}

\subjclass[2010]{11T06, 11T55}

\keywords{Congruence preserving function, polynomial function, polynomials over finite fields, residue class ring}

\begin{abstract}
In this paper, as an analogue of the integer case, we define congruence preserving functions over the
residue class rings of polynomials over finite fields. 
We establish a counting formula for such congruence preserving functions, 
determine a necessary and sufficient condition under which all congruence preserving functions are also polynomial functions, 
 and  characterize such functions.
\end{abstract}

\maketitle

\section{Introduction}

\subsection{Motivation}

Let $m$ and $n$ be positive integers. In \cite{Chen1995},  Chen gave the following definition.

\begin{definition}[\cite{Chen1995}]   \label{def:Chen}
A function  $f: \Z/n\Z \to \Z/m\Z$ is said to be a 
\textit{polynomial function} if it is representable by a polynomial $F\in \Z[x]$, that is, 
$$
f(a)\equiv F(a)\pmod m, \text{for all\ \ }  a=0,1,\ldots,n-1, 
$$
where $a$ is considered as in $\Z$ when evaluating $F(a)$.
\end{definition}

Chen also extended the above concept to multivariables in \cite{Chen1996}.  
The concept of congruence preserving function from $\Z/n\Z$ to $\Z/m\Z$  was implied in \cite{Chen1995} 
and  was defined clearly by Bhargava in \cite{Bhargava1997-1}. 

\begin{definition}[\cite{Bhargava1997-1}]    \label{def:Bhar}
A function  $f: \Z/n\Z \to \Z/m\Z$ is said to be \textit{congruence preserving} 
if for all $a, b \in \{0,1,\ldots,n-1\}, f(a)\equiv f(b) \pmod d$ whenever $a\equiv b \pmod d$ and $d$ divides $m$.
\end{definition}

 It is easy to see that any polynomial function $f: \Z/n\Z \to \Z/m\Z$ is congruence preserving.
  Chen \cite{Chen1995} posed the problem of determining all pairs $(n,m)$ for which the converse is also true. 
  Bhargava  gave a complete answer to Chen's problem by determining all such pairs in \cite[Theorem 1]{Bhargava1997-1} 
  and called such pairs $(n,m)$ as \textit{Chen pairs}. 
  
 Recently, in \cite[Theorem 1.7]{CGG2016}  C\'{e}gielski, Grigorieff and Guessarian characterized all the congruence preserving functions from $\Z/n\Z$ to $\Z/m\Z$ 
 by using binomial functions. 

In \cite{LS2018}, we generalized the notion of polynomial function to the case of  Dedekind domains.
 In this paper, we want to generalize congruence preserving functions to the case of polynomial rings over finite fields,  
and then  determine all the Chen pairs and establish a characterization for such functions 
 by following the  strategies in \cite{Bhargava1997-1} and in \cite{Frisch1993} respectively.

\subsection{Our situation}

Let $\F_q$ be the finite field of $q$ elements, where $q$ is a power of a prime $p$.
Denote by $\A=\F_q[t]$ the polynomial ring of one variable over $\F_q$. 
For any non-zero $f \in \A$, define $|f| = q^{\deg f}$.  

For any non-constant polynomial $f\in \A$, let $\A_f$ be the residue class ring of $\A$ modulo $f$,
and let
$$
\mathcal{A}_f = \{h\in \A: \, \deg h < \deg f \} \cup \{0\}.
$$
Note that $\mathcal{A}_f $ is a complete set of the representatives of the residue classes modulo $f$.
For any $h \in \cA_f$, denote by $\bar{h}$ the residue class of $h$ modulo $f$. 

From now on, $f, g \in \A$ are two non-constant polynomials.  
Definitions~\ref{def:Chen} and \ref{def:Bhar} can be generalized as follows. 

\begin{definition}[\cite{LS2018}]
A function  $\sigma: \A_f \to \A_g$ is said to be a 
\textit{polynomial function} if it is representable by a polynomial $F \in \A[x]$, that is, 
$$
\sigma(\bar{h}) = F(h) \textrm{ mod $g$} \quad \textrm{for any $h\in \mathcal{A}_f$}. 
$$
\end{definition}

\begin{definition}
A function $\sigma: \A_f \to \A_g$ is said to be \textit{congruence preserving}, 
if for any $h_1, h_2 \in \mathcal{A}_f, \sigma(\bar{h}_1)\equiv \sigma(\bar{h}_2) \pmod{h}$ 
whenever $h_1\equiv h_2 \pmod{h}$ and $h$ divides $g$.
\end{definition}

If every congruence preserving function from $\A_f$ to $\A_g$ is also a polynomial function, 
then we say that  $(f,g)$ is a \textit{Chen pair}. 

In this paper, we give a counting formula for the number of congruence preserving functions from $\A_f$ to $\A_g$, 
determine all the Chen pairs $(f,g)$, and characterize such functions. 

We first present the main results and then prove them 
in Sections~\ref{sec:count-P}, \ref{sec:Chen-P}, \ref{sec:density-P} and \ref{sec:char-P} respectively.

\subsection{Counting congruence preserving functions}

Recall that $f, g \in \A$ are two non-constant polynomials.   
In the sequel, assume that the prime factorization of $g$ is:
\begin{equation}    \label{eq:g-factor}
 g = \alpha P_{1}^{e_{1}}\cdots P_{r}^{e_{r}},  
\end{equation}
where $\alpha \in \F_q^*$,  each $e_i \ge 1$ and each $P_i$ is a monic irreducible polynomial of positive degree over $\F_q, i= 1, \ldots, r$. 
  
\begin{theorem}   \label{thm:count}
The number $M(f,g)$ of congruence preserving functions  from $\A_f$ to $\A_{g}$ is given by 
$$
M(f,g)=\frac{|g|^{q^{n}}}{\prod_{i=1}^{r}|P_{i}|^{(q-1) \sum_{k=1}^{n-1}q^{k}\min\{e_{i},\lfloor\frac{k}{d_{i}}\rfloor\}}},
$$
where $n = \deg f$, and $d_{i} = \deg P_i, i=1, \ldots, r$.
\end{theorem}

Noticing that the total number of functions from $\A_f$ to $\A_g$ is $|g|^{q^{n}}$, 
as a direct consequence we have: 

\begin{corollary} 
Every function $\sigma: \A_f \to \A_g$ is congruence preserving if and only if $\deg f \leq \deg P_{i}$ for any $1\leq i \leq r$.  
\end{corollary}

\begin{remark} 
By  \cite[Theorem 3.1]{LS2018}, we know that every function $\sigma: \A_f \to \A_g$ is a polynomial function if and only if for any $1\leq i \leq r$,
no two elements of $\mathcal{A}_f$ are congruent modulo $P_{i}$, that is, $\deg f \leq \deg P_{i}$. 
So, if every function from $\A_{f}$ to $\A_{g}$ is congruence preserving, then every such function is also a polynomial function.  
\end{remark}

\subsection{Determining Chen pairs} 
\label{sec:Chen}

As usual, $\N$ stands for the set of non-negative integers. 
To simplify the statement, we define a function
$\gamma: \A \setminus \F_q \to  \N\cup\{\infty\}$ on powers of irreducible polynomials $P$ by
if $q=2$
\begin{equation*}
\begin{split}
\gamma(P^e)
& = \left\{\begin{array}{ll}
\infty  & \textrm{if $e=1$,}\\
\infty  & \textrm{if $e=2$ and $\deg P=1$,}\\
\deg P+2  & \textrm{otherwise,}\\
\end{array}
\right.
\end{split}
\end{equation*}
if $q>2$
\begin{equation*}
\begin{split}
\gamma(P^e)
& = \left\{\begin{array}{ll}
\infty  & \textrm{if $e=1$,}\\
\deg P +1  &\textrm{otherwise},
\end{array}
\right.
\end{split}
\end{equation*}
and on all other non-constant polynomials $h$ by
$$
\gamma(h)=\min\{\gamma(Q_{1}^{k_1}),\ldots,\gamma(Q_{s}^{k_s})\},
$$
where $h$ has the prime factorization $\beta Q_{1}^{k_1}\cdots Q_{s}^{k_s}, \beta \in \F_q^*$. 

The characterization of a Chen pair is given below. 

\begin{theorem}   \label{thm:Chen}
The pair $(f,g)$ is a Chen pair if and only if $\deg f<\gamma(g)$.
\end{theorem}

Specializing Theorem~\ref{thm:Chen} to $f=g$ gives the following corollary directly. 

\begin{corollary}\label{cor:gg} 
The pair $(g,g)$ is a Chen pair if and only if 
\begin{itemize}
\item[(1)] when $q=2$, $x^3 \nmid g, (x+1)^3 \nmid g$, and for all irreducible polynomials $P$ with $\deg P\geq 2, P^2\nmid g$; or 

\item[(2)] when $q>2$, $g$ is a square-free polynomial.
\end{itemize}
\end{corollary}

Moreover, we can determine the natural density of Chen pairs $(g,g)$, that is, the limit 
\begin{equation*}
\rho = \lim_{m \to \infty} \frac{\# \{g\in \A:\, \textrm{$1 \le \deg g \le m$, $(g,g)$ is a Chen pair}\}}{\# \{g \in \A:\, 1 \le \deg g \le m\}}. 
\end{equation*}
The following result suggests that most of the pairs $(g,g)$ are Chen pairs. 

\begin{theorem}   \label{thm:density}
The natural density of Chen pairs $(g,g)$ is given by 
\begin{equation*}
\begin{split}
\rho 
& = \left\{\begin{array}{ll}
\frac{49}{72}  & \textrm{if $q=2$,}\\
\\
\frac{q-1}{q}  &\textrm{if $q>2$.}
\end{array}
\right.
\end{split}
\end{equation*}
\end{theorem}

Corollary~\ref{cor:gg} implies that $(g,g)$ is a Chen pair if and only if  $\gamma(g)=\infty$, that is, 
if and only if  $(f,g)$ is a Chen pair for all $f\in \A$ by Theorem~\ref{thm:Chen}. 
So, in some sense we can say that most of the pairs $(f,g)$ in $\A \times \A$ are Chen pairs.

\subsection{Characterizing congruence preserving functions } 
\label{sec:char}

We first consider the special case when $g=P^{e}$ with $e \ge 1$ and $P$ irreducible polynomial of degree $d \ge 1$.

 Let $\{b_{0},b_{1},\cdots,b_{q^{d}-1}\}$ be a fixed ordering of the polynomials in $\A$ of degree less than $d$ such that $b_{0}=0,b_{1}=1$ 
 and $\deg b_{i}\leq \deg b_{j}$ for any $1\leq i \leq j$. 
 Then, for every $k\in \mathbb{N}$, define
$$ 
b_{k}=b_{l_{0}}+b_{l_{1}}P+\ldots+b_{l_{m}}P^{m},
$$
where $\sum_{i=0}^{m}l_{i}q^{di}$ is the $q^{d}$-adic expansion of $k$, and define the binomial polynomials over $\F_q(t)$: 
$$
Q_0(x) =  1, \qquad Q_{k}(x)=\frac{\prod_{j=0}^{k-1}(x-b_{j})}{\prod_{j=0}^{k-1}(b_{k}-b_{j})} \quad \textrm{if $k>0$}. 
$$
Clearly, $Q_{1}(x)=x$. 
By \cite[Theorem 3.3]{Wagner1971}, we know that for any $k \in \N$ and any $h \in \A$, 
$Q_k(h)$ is $P$-integral (that is, its valuation at $P$ is non-negative), 
and so its reduction modulo  $P^e$ is well-defined.  
Then, for any $k \in \N$ we define the function $B_k: \A_f \to \A_{P^{e}}$ by 
$$
B_k(\bar{h}) = Q_k(h)  \textrm{ mod $P^e$}  \quad \textrm{for any $h\in \mathcal{A}_f$}.
$$
In particular, $B_0(\bar{h}) = 1  \textrm{ mod $P^e$}$ for any $h \in \cA_f$. 

\begin{theorem}\label{thm:char} 
A function $\sigma: \A_f \to \A_{P^{e}}$ is congruence preserving if and only if 
there is a unique sequence $\{c_{0},c_{1},\cdots,c_{q^{n}-1}\}$ of elements of $\A_{P^{e}}$ such that
\begin{equation}   \label{eq:char}
\sigma = \sum_{k=0}^{q^{n}-1}c_{k}B_{k},
\end{equation} 
and for each $k=1,\ldots, q^{n}-1$, $c_{k}$ is in the subgroup generated by $P^{\mu(k)}\pmod{P^{e}}$ 
with $\mu(k)=\lfloor\frac{\log_{q}k}{d}\rfloor$. Here, $n= \deg f, d= \deg P$.  
\end{theorem}

Finally, the general case follows directly from Theorem~\ref{thm:char} and Lemma~\ref{lem:local-global}~(1) below. 

\begin{corollary}\label{cor:char} 
Assume that the polynomial $g$ has the prime factorization as in \eqref{eq:g-factor}. 
Then, a function $\sigma: \A_f \to \A_g$ is congruence preserving 
if and only if for every $i= 1, \ldots, r$, $\sigma_{i}$ has an expression as in \eqref{eq:char},  
where $\sigma_{i}$ is the reduction of $\sigma$ modulo $ P_{i}^{e_{i}}$.
\end{corollary}

Different from the integer case (see \cite[Theorem 1.7]{CGG2016}), 
in the case of $\A$ we don't have a uniform characterization for such congruence preserving functions. 
The reason is that there is no sequence over $\A$ such that it is a $P$-sequence (see Definition~\ref{def:Pseq} below) 
for any irreducible polynomial $P$ in $\A$; see, for instance, \cite[Examples 2.5 and 2.6]{Frisch1999}.

\section{Proof of Theorem~\ref{thm:count}}
\label{sec:count-P}

 The following lemma is an analogue of  \cite[Proposition 1]{Bhargava1997-1} proved via the Chinese Remainder Theorem. 
 It implies that we only need to consider the special case when $g=P^{e}$ with $e \ge 1$ and $P$ irreducible polynomial. 
 We omit its proof. 

\begin{lemma}  \label{lem:local-global}
Let $\sigma: \A_f\to\A_g$ be a function. 
Assume that $g$ has the prime factorization as in \eqref{eq:g-factor}. 
For $1\leq i\leq r$, let $\sigma_{i}: \A_f\to\A_{P_{i}^{e_{i}}}$ be the functions obtained by taking the function values of $\sigma$ modulo $P_{i}^{e_{i}}$. 
Then  
\begin{itemize}
\item[(1)] $\sigma: \A_f\to\A_g$ is congruence preserving if and only if $\sigma_{i}: \A_f\to\A_{P_{i}^{e_{i}}}$ is for each $1\leq i\leq r$.

\item[(2)] $\sigma: \A_f\to\A_g$ is a polynomial function if and only if $\sigma_{i}: \A_f\to\A_{P_{i}^{e_{i}}}$ is for each $1\leq i\leq r$.

\item[(3)] $(f,g)$ is a Chen pair if and only if $(f,P_{i}^{e_{i}})$ is  for each $1\leq i\leq r$.

\end{itemize}
\end{lemma}

We first establish a counting formula for the number of congruence preserving functions from $\A_f$ to $\A_{P^e}$.

\begin{lemma}\label{lem:MfPe}
The number $M(f,P^e)$ of congruence preserving functions from $\A_f$ to $\A_{P^e}$ is given by 
$$
M(f,P^e)=|P|^{eq^{n}-(q-1)\Sigma_{k=1}^{n-1}q^{k}\min\{e,\lfloor\frac{k}{d}\rfloor\}},
$$
where $n=\deg f, d = \deg P$.  
\end{lemma}

\begin{proof}
By definition, since $\cA_f = \cA_{t^n}$, we have $M(f,P^e)= M(t^n,P^e)$. 
So, it is equivalent to compute $M(t^n,P^e)$. 

If $n=1$, then $\A_t = \F_q$. 
Note that every function from $\F_q$ to $\A_{P^e}$ is a polynomial function (see, for instance, \cite[Theorem 3.1]{LS2018}), 
and thus a congruence preserving function. 
So, we have 
\begin{equation}   \label{eq:n=1}
M(t,P^e) = |P|^{eq}, 
\end{equation}
which is the desired result when $n=1$. 

Now, assume that $n > 1$. 
To compute $M(t^n,P^e)$, we want to first obtain a recursive relation between $M(t^n,P^e)$ and $M(t^{n-1},P^e)$.

Notice that for any congruence preserving function from $\A_{t^n}$ to $\A_{P^e}$,  
its restriction to $\A_{t^{n-1}}$ gives a congruence preserving function from  $\A_{t^{n-1}}$ to $\A_{P^e}$. 
We therefore need to determine the number of ways a given congruence preserving function $\sigma:\A_{t^{n-1}}\to\A_{P^e}$ 
can be extended to a congruence preserving function from $\A_{t^n}$ to $\A_{P^e}$. 
This is equal to the number of ways $\sigma(\bar{h})$ with $h$ of degree $n-1$ can be assigned while preserving the necessary congruences. 
That is, $\sigma(\bar{h})$ can take values from $\A_{P^e}$, but if $h\equiv a \pmod{P^{j}}$, 
where $a\in\cA_{t^n}$ and $j \leq e$, then we must have
$\sigma(\bar{h})\equiv \sigma(\bar{a}) \pmod{P^{j}}$.

So, we need to know the largest $j$ such that $\sigma(\bar{h})$ is determined modulo $P^{j}$. 
This largest $j$ is easily seen to be given by $\min\{e,\lfloor\frac{n-1}{d}\rfloor\}$.

It follows that if $\sigma$ is to remain congruence preserving when extended to $\A_{t^n}$, 
then $\sigma(\bar{h})$ can take on a total of exactly 
$$
\frac{|P|^e}{|P|^{\min\{e,\lfloor\frac{n-1}{d}\rfloor\}}}
$$ 
values. 
Note that the number of such $h$ with degree $n-1$ is $(q-1)q^{n-1}$, we therefore have the relation
$$
M(t^n,P^e)=(|P|^{e-\min\{e,\lfloor\frac{n-1}{d}\rfloor\}})^{(q-1)q^{n-1}}M(t^{n-1},P^e).
$$
Using this relation repeatedly, together with \eqref{eq:n=1}, yields the desired result.
\end{proof}

Now, it is easy to prove Theorem~\ref{thm:count}. 

\begin{proof}[Proof of Theorem~\ref{thm:count}]
Given a function $\sigma: \A_f\to\A_g$, 
 by reducing the values of $\sigma$ modulo $P_{i}^{e_{i}}$ for $1\leq i\leq r$, we obtain functions $\sigma_{i}: \A_f\to\A_{P_{i}^{e_{i}}}$. 
 Conversely, given functions $\sigma_{i}: \A_f\to\A_{P_{i}^{e_{i}}}$ for $1\leq i\leq r$, by the Chinese Remainder Theorem there is a unique function $\sigma: \A_f\to\A_g $ such that $\sigma$ reduces to $\sigma_{i}$ modulo $P_{i}^{e_{i}}$ for each $i$. 
 This observation, together with Lemma~\ref{lem:local-global} (1),  shows 
$$
M(f,g)=\prod_{i=1}^{r}M(f,P_{i}^{e_{i}}),
$$ 
then by Lemma~\ref{lem:MfPe}, we derive the desired counting formula for such congruence preserving functions.  
\end{proof}

\section{Proof of Theorem~\ref{thm:Chen}}
\label{sec:Chen-P}

The strategy to prove Theorem~\ref{thm:Chen} is to compare the number of congruence preserving functions 
with the number of polynomial functions from $\A_f$ to $\A_{g}$. 
By Lemma~\ref{lem:local-global} (3), we in fact only need to consider functions from $\A_f$ to $\A_{P^e}$ 
with $e \ge 1$ and irreducible polynomial $P$.  

We first recall a counting formula, given in \cite[Theorem 4.4]{LS2018},  for the number of polynomial functions from $\A_f$ to $\A_g$. 

Write $\F_q=\{a_{0}=0,a_{1},\ldots,a_{q-1}\}$, and for every $k\in \mathbb{N}$, let
$$ 
a_{k}=a_{l_{0}}+a_{l_{1}}t+\ldots+a_{l_{m}}t^{m} \in \A,
$$
where $\sum_{i=0}^{m}l_{i}q^{i}$ is the $q$-adic expansion of $k$. This gives us a one-to-one correspondence between $\mathbb{N}$ and $\A$. 
Then, as an analogue of factorials of non-negative integers, one can define factorials for polynomials in $\A$ by 
$$
\prod_{i=0}^{k-1}(a_{k}-a_{i}), \quad k \ge 1, 
$$
and the factorial is $1$ when $k=0$; see \cite{LS2017} for another analogue.

The following is a special case in \cite[Theorem 4.4]{LS2018}. 

\begin{theorem}  \label{thm:Nfg}
The number $N(f,g)$ of polynomial functions  from $\A_f$ to $\A_{g}$ is given by 
$$
N(f,g)=\frac{q^{q^{n}\deg g}}{\prod_{k=1}^{q^{n}-1}q^{\deg\gcd(g, \prod_{i=0}^{k-1}(a_{k}-a_{i}))}}, 
$$
where $n = \deg f$.  
\end{theorem}

When $g=P^e$, we have: 

\begin{corollary} \label{cor:NfPe}
The number $N(f,P^e)$ of polynomial functions  from $\A_f$ to $\A_{P^e}$ is given by 
$$
N(f,P^e)=|P|^{eq^{n}-\sum_{k=1}^{q^{n}-1}\min\{e,\lfloor\frac{k}{q^d}\rfloor+\lfloor\frac{k}{q^{2d}}\rfloor+\cdots\}},
$$
where $n=\deg f, d = \deg P$.  
\end{corollary}

\begin{proof}
The desired result follows by 
substituting $g=P^e$ in Theorem~\ref{thm:Nfg} and using the fact from \cite[Example 3]{Bhargava1997-2} that 
$$
\gcd(P^e, \prod_{i=0}^{k-1}(a_{k}-a_{i})) = P^{\min\{e,\lfloor\frac{k}{q^d}\rfloor+\lfloor\frac{k}{q^{2d}}\rfloor+\cdots\}}. 
$$
\end{proof}

Recalling the function $\gamma$ defined in Section~\ref{sec:Chen}, we first determine a condition when the pair $(f,P^e)$ is a Chen pair. 

\begin{lemma}   \label{lem:Chen}
The pair $(f,P^{e})$ is a Chen pair if and only if $\deg f < \gamma(P^e)$.  
\end{lemma}

\begin{proof}
Let $n=\deg f, d = \deg P$.  By definition, $(f,P^{e})$ is a Chen pair if and only if 
$$
M(f,P^e) = N(f,P^e),  
$$
which, together with Lemma~\ref{lem:MfPe} and Corollary~\ref{cor:NfPe}, is equivalent to 
\begin{equation}  \label{eq:fPe1}
|P|^{eq^{n}-(q-1)\sum_{k=1}^{n-1}q^{k}\min\{e,\lfloor\frac{k}{d}\rfloor\}}
= |P|^{eq^{n}-\sum_{k=1}^{q^{n}-1}\min\{e,\lfloor\frac{k}{q^d}\rfloor+\lfloor\frac{k}{q^{2d}}\rfloor+\cdots\}}. 
\end{equation}
Note that 
\begin{equation}   \label{eq:MfPe2}
\begin{split}
(q-1)\sum_{k=1}^{n-1}q^{k}\min\{e,\lfloor\frac{k}{d}\rfloor\} 
& =\sum_{k=1}^{n-1}\sum_{j=q^{k}}^{q^{k+1}-1}\min\{e,\lfloor\frac{\log_{q}j}{d}\rfloor\} \\
& =\sum_{k=1}^{q^{n}-1}\min\{e,\lfloor\frac{\log_{q}k}{d}\rfloor\}.
\end{split}
\end{equation}
So, by \eqref{eq:MfPe2} the condition~\eqref{eq:fPe1} is equivalent to 
\begin{equation}   \label{eq:fPe2}
\sum_{k=1}^{q^{n}-1}\min\{e,\lfloor\frac{\log_{q}k}{d}\rfloor\}
= \sum_{k=1}^{q^{n}-1}\min\{e,\lfloor\frac{k}{q^d}\rfloor+\lfloor\frac{k}{q^{2d}}\rfloor+\cdots\}. 
\end{equation}

On the other hand, for any integer $k \ge 1$  we have 
\begin{equation*}
\lfloor\frac{\log_{q}k}{d}\rfloor = \sum_{i=1}^{\infty}\min\{1,\lfloor\frac{k}{q^{di}}\rfloor\}\leq \lfloor\frac{k}{q^{d}}\rfloor+\lfloor\frac{k}{q^{2d}}\rfloor+\cdots,\end{equation*}
and the equality occurs in the above if and only if $k<2q^{d}$.
Thus, the inequality
\begin{equation}\label{eq:fPe3}
\min\{e,\lfloor\frac{\log_{q}k}{d}\rfloor\} \leq\min\{e,\lfloor\frac{k}{q^d}\rfloor+\lfloor\frac{k}{q^{2d}}\rfloor+\cdots\} 
\end{equation}
also holds, and the equality in \eqref{eq:fPe3} occurs if and only if $k<2q^{d}$ or $k\geq q^{ed}$.

Therefore, using \eqref{eq:fPe3}, the condition~\eqref{eq:fPe2} holds if and only if $k<2q^{d}$ or $k\geq q^{ed}$
for all $1\leq k \leq q^{n}-1$. This is true exactly when $q^{ed}\leq 2q^{d}$ or $q^{n}\leq 2q^{d}$. 
In other words, the condition~\eqref{eq:fPe2} holds  if and only if 
$$
\textrm{$e=1$, or $q=2,e=2,d=1$, or $q^n\leq 2q^d$.} 
$$
By the definition of the function $\gamma$, this is equivalent to the condition $n<\gamma(P^e)$.  
We thus complete the proof. 
\end{proof}

Now, we are ready to prove Theorem~\ref{thm:Chen}. 

\begin{proof}[Proof of Theorem~\ref{thm:Chen}] 
Assume that the polynomial $g$ has the prime factorization as in \eqref{eq:g-factor}. 
By Lemma~\ref{lem:local-global} (3) and Lemma~\ref{lem:Chen}, 
we know that $(f,g)$ is a Chen pair if and only if $\deg f < \gamma(P_{i}^{e_{i}})$ for each $1\leq i\leq r$, 
that is,  if and only if $\deg f < \min\{\gamma(P_{1}^{e_{1}}),\ldots,\gamma(P_{r}^{e_{r}})\}=\gamma(g)$. 
This proves Theorem~\ref{thm:Chen}. 
\end{proof}

\section{Proof of Theorem~\ref{thm:density}}
\label{sec:density-P}

From Corollary~\ref{cor:gg}, we know that there are two cases 
for $(g,g)$ being a Chen pair depending on $q=2$ or $q>2$. 

When $q>2$, it suffices to count square-free polynomials. 
This in fact is well-known (including the case $q=2$); see \cite[Proposition 2.3]{Rosen2000}. 

\begin{lemma} \label{square-free}
Let $S(n)$ be the number of all monic square-free polynomials of degree  $n$ in $\A$. Then
\begin{equation*}
\begin{split}
S(n)
& = \left\{\begin{array}{ll}
1& \textrm{if $n=0$,}\\
q & \textrm{if $n=1$,}\\
q^{n}-q^{n-1}  & \textrm{if $n\geq 2$.}\\
\end{array}
\right.
\end{split}
\end{equation*}
\end{lemma}

When $q=2$, it needs more work, 
because we need to count polynomials satisfying the condition in Corollary~\ref{cor:gg}~(1).

\begin{lemma}
\label{square-free2}
When $q=2$, let $T(n)$ be the number of all polynomials of degree $n$ in $\A$ satisfying the condition in Corollary~\ref{cor:gg}~(1). 
 Then
\begin{equation*}
\begin{split}
T(n)
& = \left\{\begin{array}{ll}
1& \textrm{if $n=0$ ,}\\
2& \textrm{if $n=1$ ,}\\
4& \textrm{if $n=2$ ,}\\
6& \textrm{if $n=3$ ,}\\
\frac{1}{9} \big( 2^{n-3}7^2+(-1)^{n-1}(3n-13) \big) & \textrm{if $n\geq4$.}
\end{array}
\right.
\end{split}
\end{equation*}
\end{lemma}

\begin{proof}
For $n=0,1,2,3$, by simple calculations, $T(n)=1,2,4,6$, respectively.
Now, we suppose $n\geq4$, and denote  by $U(n)$ the set of all polynomials of degree $n$ in $\A$ satisfying the condition in Corollary~\ref{cor:gg}~(1).
 Clearly, $U(n)$ can be divided into four disjoin subsets:
\begin{align*}
& U_{1}(n)=\{g :\,  \textrm{$g$ square-free, $\deg g =n$} \}, \\
& U_{2}(n)=\{x^{2}g :\, \textrm{$g$ square-free, $\deg g =n-2, x\nmid g$}\},\\
 & U_{3}(n)=\{(x+1)^{2}g :\, \textrm{$g$ square-free, $\deg g= n-2, x+1\nmid g$}\},\\
  & U_{4}(n)=\{x^{2}(x+1)^{2}g :\, \textrm{$g$ square-free, $\deg g= n-4, x\nmid g, x+1\nmid g$}\}.
\end{align*}
So, we have
$$
T(n)= \# U(n) = \# U_{1}(n)+\# U_{2}(n)+\# U_{3}(n)+\# U_{4}(n).
$$
Then, it remains to compute the sizes $\# U_{i}(n), i=1,2,3,4$.  

Firstly, by Lemma \ref{square-free}, we know that 
$$
\# U_{1}(n)=S(n)=2^{n-1}.
$$

For $U_{2}(n)$, we have
\begin{align*}
U_{2}(n) 
= & \{x^{2}g :\, \textrm{$g$ square-free,  $\deg g =n-2$} \} \\
&-\{x^{3}g :\, \textrm{$g$ square-free, $\deg g = n-3$, $x\nmid g$}\},
\end {align*}
which gives the following recursive formula
$$
\# U_{2}(n)=S(n-2)-\# U_{2}(n-1).
$$
So, we obtain
$$
\# U_{2}(n)=\sum_{i=0}^{n-2}(-1)^{i}S(n-2-i).
$$
Using Lemma \ref{square-free}, we further have
$$
\# U_{2}(n)=\frac{1}{3}\big( 2^{n-2} + (-1)^{n-1} \big).
$$

For $U_{3}(n)$, by symmetry, we have
$$
\# U_{3}(n)=\# U_{2}(n)= \frac{1}{3}\big( 2^{n-2} + (-1)^{n-1} \big).
$$

For $U_{4}(n)$, we first have 
\begin{align*}
\# U_{4}(n)   = & \# \{g :\, \textrm{$g$ square-free, $\deg g= n-4, x\nmid g, (x+1)\nmid g$}\} \\
=& \# \{g :\, \textrm{$g$ square-free, $\deg g= n-4$}\}\\
&- \# \{g :\, \textrm{$g$ square-free, $\deg g= n-4,\ x\mid g, (x+1)\nmid g$}\}\\
& - \# \{g :\, \textrm{$g$ square-free, $\deg g= n-4,\ x\nmid g, (x+1)\mid g$}\}\\
& - \#  \{g :\, \textrm{$g$ square-free, $\deg g= n-4,\ x\mid g, (x+1)\mid g$}\}\\
= & \# \{g :\, \textrm{$g$ square-free, $\deg g= n-4$}\}\\
&- \# \{g :\, \textrm{$g$ square-free, $\deg g= n-5,\ x\nmid g, (x+1)\nmid g$}\}\\
& - \# \{g :\, \textrm{$g$ square-free, $\deg g= n-5,\ x\nmid g, (x+1)\nmid g$}\}\\
& - \# \{g :\, \textrm{$g$ square-free, $\deg g= n-6,\ x\nmid g, (x+1)\nmid g$}\},
\end {align*}
which implies that
$$
\# U_{4}(n)=S(n-4)-2\# U_{4}(n-1)-\# U_{4}(n-2).
$$
So, we get
$$
\# U_{4}(n)=\sum_{i=0}^{n-4}(-1)^{i}(i+1)S(n-4-i).
$$
Using Lemma \ref{square-free} again, we obtain 
$$
\# U_{4}(n)=\frac{1}{9} \big( 2^{n-3} +(-1)^{n-1}(3n-19) \big).
$$

 Finally, collecting the above calculations, we have 
 $$
 T(n)=\frac{1}{9} \big( 2^{n-3}7^2+(-1)^{n-1}(3n-13) \big), \quad n \ge 4.
 $$
 This completes the proof.
\end{proof}

We are now ready to prove Theorem~\ref{thm:density}. 

\begin{proof}[Proof of Theorem~\ref{thm:density}] 
When $q=2$, by Corollary~\ref{cor:gg}~(1) and Lemma \ref{square-free2}, we have 
\begin{align*}
\rho & = \lim_{m \to \infty} \frac{\# \{g\in \A:\, \textrm{$1 \le \deg g \le m$, $(g,g)$ is a Chen pair}\}}{\# \{g \in \A:\, 1 \le \deg g \le m\}} \\ 
&=\lim_{m \to \infty}\frac{T(1)+ T(2) + \cdots + T(m)}{2^{m+1} - 2}\\
&=\frac{49}{72}.
\end{align*}

Besides, when $q>2$, using Corollary~\ref{cor:gg}~(2) and Lemma \ref{square-free}, we obtain 
\begin{align*}
\rho &=\lim_{m \to \infty}\frac{(q-1)(S(1)+\cdots+S(m))}{q^{m+1} - q}\\
&=\lim_{m \to \infty}\frac{(q-1)q^{m}}{q^{m+1} -q }\\
&=\frac{q-1}{q}.
\end{align*}
\end{proof}

\section{Proof of Theorem~\ref{thm:char}}
\label{sec:char-P}

In this section, $P$ is always an irreducible polynomial of degree $d$ in $\A$. 
Let $v_{P}$ be the additive valuation of $\A$ at $P$.
By convention, put $v_P(0) = \infty$. 

We first recall the notion of $P$-sequence and homogeneous $P$-sequence over $\A$ as in Frisch's PhD thesis \cite[Definition 2.3]{Frisch1993}. 

\begin{definition}  \label{def:Pseq}
A sequence $\{ u_{i}\}$ (finite or infinite) over $\A$ is a $P$-sequence, if for any $m \in \N$ and all $i,j$,  
$$
 v_{P}(u_{i} -u_{j})\geq m \, \Longleftrightarrow \, q^{dm} \mid i-j,
$$
 and a homogeneous $P$-sequence if in addition 
 $$
 \forall \,  i, m \in \N,  \quad  v_{P}(u_{i})\geq m  \, \Longleftrightarrow \,  q^{dm} \mid i.
 $$
\end{definition}

In particular, any $P$-sequence $\{ u_{i}\}$ is also a homogeneous $P$-sequence if $u_0=0$. 

For the sequence $\{b_k\}_{k=0}^{\infty}$ defined in Section~\ref{sec:char}, 
it is easy to see that it is a homogeneous $P$-sequence, that is, for any $i,j,m \in \N$, 
 $$
 v_{P}(b_{i} -b_{j}) \geq m \, \Longleftrightarrow \,  q^{dm} \mid i-j.
 $$

The following result is a key lemma in the proof. 
It is a special case in \cite[Lemma 2.24]{Frisch1993}. 
Here we omit its proof. 

\begin{lemma}[\cite{Frisch1993}]  \label{lem:2.24}
For any integer $k \ge 1$, let $\{b_{i}^{'}\}_{i=1}^{k}$ be a re-ordering of $\{b_{i}\}_{i=1}^{k}$ by increasing valuation at $P$, 
and $\{u_i\}_{i=1}^{k}$ a $P$-sequence over $\A$. 
Then, for any non-empty subset $J \subseteq \{1,2,\ldots,k\}$ and any $h_1, h_2 \in \A$, we have
$$
v_{P} \big(\prod_{j\in J}(h_1 - u_{j})-\prod_{j\in J}(h_2 - u_{j}) \big) \geq v_{P}(h_1 - h_2) + \sum_{i=1}^{|J|-1}v_{P}(b_{i}^{'}).
$$
\end{lemma}

Recall that  $\mu(k)=\lfloor\frac{\log_{q}k}{d}\rfloor, k \ge 1$.  
The following lemma is a direct consequence of Lemma~\ref{lem:2.24}. 
We need it to prove the congruence preserving property. 

\begin{lemma}  \label{lem:2.26} 
For any integer $k \ge 1$ and any polynomials $h_1, h_2 \in \A$, we have
$$
v_{P} \big(\prod_{j=0}^{k-1}(h_1 - b_{j})-\prod_{j=0}^{k-1}(h_2 - b_{j}) \big) 
\geq  v_{P}(h_1 - h_2) + \sum_{j \ge 1} \lfloor \frac{k}{q^{dj}} \rfloor  - \mu(k).
$$
\end{lemma}

\begin{proof} 
The case $k=1$ is trivial by noticing $b_0 = 0$. 
We assume that $k \ge 2$.  
In Lemma~\ref{lem:2.24}, choosing $u_i = b_{i-1}, i =1, 2, \ldots, k$ and $J=\{1,2, \ldots, k\}$, we obtain  
$$
v_{P} \big(\prod_{j=0}^{k-1}(h_1 - b_{j})-\prod_{j=0}^{k-1}(h_2 - b_{j}) \big) \geq v_{P}(h_1 - h_2) + \sum_{i=1}^{k-1}v_{P}(b_{i}^{'}).
$$
Besides, we have 
\begin{align*}
\sum_{i=1}^{k-1}v_{P}(b_{i}^{'})   \ge \sum_{i=1}^{k}v_{P}(b_{i}) - v_P(b_{k}^{'})  
 =  \sum_{j \ge 1} \lfloor \frac{k}{q^{dj}} \rfloor  - \mu(k), 
\end{align*}
where the equality follows from \cite[Lemma 2.7 (a)]{Frisch1996} and the fact $v_P(b_{k}^{'})=\mu(k)$.   
This in fact completes the proof. 
\end{proof}

As in \cite[Proposition 1.6]{CGG2016}, we have: 

\begin{lemma} \label{lem:rep}
For every function $\sigma: \A_f\to\A_{P^e}$, there is a unique sequence $\{c_{0},c_{1},\cdots,c_{q^{n}-1}\}$ of elements of $\A_{P^{e}}$ such that
$$
\sigma = \sum_{k=0}^{q^{n}-1}c_{k}B_{k}, 
$$ 
where $n=\deg f$.  
\end{lemma}

\begin{proof} 
It is easy to see that $ \mathcal{A}_{f} = \{b_{0}=0,b_{1},\cdots,b_{q^n - 1}\}$. 
We also note that $B_{k}(\bar{b}_{k})=1 \textrm{ mod $P^e$}$, and $B_{k}(\bar{b}_{i})=0 \textrm{ mod $P^e$}$ for any $i < k $. 
Then, $\sigma(\bar{b}_0)=c_0$, $\sigma(\bar{b}_1) = c_0 + c_1$, and so on. 
Thus, the existence and uniqueness of the sequence $\{c_{0},c_{1},\cdots,c_{q^{n}-1}\}$ can be proved by induction. 
\end{proof}

Now, we are ready to prove Theorem~\ref{thm:char}.

\begin{proof}[Proof of Theorem~\ref{thm:char}]
We first prove the sufficiency. It is equivalent to prove that for any $1 \le k \le q^n-1$,  
the function $c_kB_k$ is congruence preserving 
when $c_{k}$ is in the subgroup generated by $P^{\mu(k)} \pmod{P^{e}}$. 
If $\mu(k) \ge e$, then $c_{k}$ is the zero element in $\A_{P^e}$, 
and so $c_kB_k$ is the zero function and automatically congruence preserving. 
We now suppose $\mu(k) < e$. 
Let $c_k^* \in \A$ be an arbitrary representative of $c_k \in \A_{P^e}$. 
By definition, it suffices to show that for any $h_1,h_2 \in \A$, 
$$
v_P(c_k^* Q_k(h_1) - c_k^* Q_k(h_2))  \ge v_P(h_1 - h_2). 
$$
Indeed, applying Lemma~\ref{lem:2.26} and \cite[Lemma 2.7 (b)]{Frisch1996} and noticing $\mu(k) < e$, 
we have 
\begin{align*}
v_P(c_k^* Q_k(h_1) - c_k^* Q_k(h_2))  
& \ge v_P(c_k^*) + v_{P}(h_1 - h_2)   - \mu(k) \\
& \ge v_{P}(h_1 - h_2), 
\end{align*}
which implies the sufficiency. 

Finally, we prove the necessity by counting argument. 
By Lemma~\ref{lem:rep}, the number of congruence preserving functions having the form \eqref{eq:char} is 
equal to 
 \begin{align*}
 \prod_{k=0}^{q^{n}-1}\#\{c_{k}\}
&=|P|^e \prod_{k=1}^{q^{n}-1}|P|^{e-\min\{e,\mu(k)\}}\\
&=|P|^{e+\sum_{k=1}^{q^{n}-1}(e-\min\{e,\lfloor\frac{\log_{q}k}{d}\rfloor\})}\\
&=|P|^{eq^{n}-\sum_{k=1}^{q^{n}-1}\min\{e,\lfloor\frac{\log_{q}k}{d}\rfloor\}}\\
&=|P|^{eq^{n}-(q-1)\Sigma_{k=1}^{n-1}q^{k}\min\{e,\lfloor\frac{k}{d}\rfloor\}},
\end{align*}
where the last equality comes from \eqref{eq:MfPe2}. 
This coincides with Lemma~\ref{lem:MfPe}. 
We thus complete the proof. 
\end{proof}

\section*{Acknowledgement}

The authors are grateful to  Professor Sophie Frisch for valuable discussions.  
For the research, the first author was supported by the National Science Foundation of China Grant No. 11526119 
and the Scientific Research Foundation of Qufu Normal University No. BSQD20130139, 
and the second author was supported by a Macquarie University Research Fellowship.

\end{document}